\documentclass{article}

\usepackage{indentfirst}
\usepackage{PRIMEarxiv}
\usepackage{amsmath,amssymb,amsthm}
\usepackage[ruled,vlined]{algorithm2e}
\usepackage[utf8]{inputenc} 
\usepackage[T1]{fontenc}    
\usepackage{hyperref}       
\usepackage{url}            
\usepackage{booktabs}       
\usepackage{amsfonts}       
\usepackage{nicefrac}       
\usepackage{microtype}      
\usepackage{lipsum}
\usepackage{graphicx}
\usepackage{color}
\usepackage{subfigure}
\usepackage{authblk}
\graphicspath{{media/}}     

\title{A Symplectic Discretization Based Proximal Point Algorithm for Convex Minimization
\thanks{This work was partially supported by grant 12288201 from NSF of China.
}}

\author[1]{Ya-xiang Yuan}
\author[ 1,2]{Yi Zhang\thanks{Corresponding author: zhangyi2020@lsec.cc.ac.cn}  }
\affil[1]{Institute of Computational Mathematics and Scientific/Engineering Computing, Academy of Mathematics and Systems
	Science, Chinese Academy of Sciences, Beijing 100190, China}
\affil[2]{University of Chinese Academy of Sciences, Beijing 100049, China}

\newcommand{\argmin}{\mathop{\arg\min}}
\newcommand{\inner}[1]{\left\langle#1\right\rangle}
\newcommand{\norm}[1]{\left\|#1\right\|}
\newcommand{\X}{\overset{\cdot}{X}}

\newcommand{\E}{\mathcal{E}}
\newcommand{\EE}{\overset{\cdot}{\mathcal{E}}}
\newcommand{\Z}{\overset{\cdot}{Z}}

\newcommand{\G}{\mathcal{G}}
\newcommand{\hh}{\mathcal{H}}
\newcommand{\dist}{\mathrm{dist}}

\newcommand{\dd}{\mathrm{d}}

\newtheorem{theorem}{Theorem}
\newtheorem{corollary}{Corollary}
\newtheorem{lemma}{Lemma}

\newtheorem{example}{Example}
\newtheorem{assumption}{Assumption}
\begin{document}
\maketitle

\begin{abstract}
The proximal point algorithm plays a central role in non-smooth convex programming. The Augmented Lagrangian Method, one of the most famous optimization algorithms, has been found to be closely related to the proximal point algorithm. Due to its importance, accelerated variants of the proximal point algorithm have received considerable attention. In this paper, we first study an Ordinary Differential Equation (ODE) system, which provides valuable insights into proving the convergence rate of the desired algorithm. Using the Lyapunov function technique, we establish the convergence rate of the ODE system. Next, we apply the Symplectic Euler Method to discretize the ODE system to derive a new proximal point algorithm, called the Symplectic Proximal Point Algorithm (SPPA). By utilizing the proof techniques developed for the ODE system, we demonstrate the convergence rate of the SPPA. Additionally, it is shown that existing accelerated proximal point algorithm can be considered a special case of the SPPA in a specific manner. Furthermore, under several additional assumptions, we prove that the SPPA exhibits a finer convergence rate.
\end{abstract}


\section{Introduction}

With the growth of high-dimensional statistics and machine learning, many optimization problems without smooth derivative have come up. The Proximal Point Algorithm (PPA) is a powerful method for dealing with such non-smooth optimization challenges. Initially, PPA was used to the study of regularization of ill-posed problems \cite{krasnoselskii1955} and was found to be closed related to the Moreau envelope \cite{moreau1965}. Over time, it has been applied to handle variational inequalities \cite{martinet1970} and non-smooth convex optimization problems \cite{rockafellar1976}. Nowadays, PPA has been applied to solve a broad range of optimization problems, as demonstrated in \cite{liu2012, wang2010, yang2013}. Notably, it was discovered that the Augmented Lagrangian Method (ALM), another important optimization algorithm, has been proven to be equivalent to the application of PPA to the Lagrangian duality problem \cite{rockafellar1976augmented, rockafellar1978}.  Recently, ALM has found diverse applications in convex optimization, as detailed in \cite{zhao2010, yang2015, li2018, li2018fused, zhang2020}. Several variations of PPA have been explored, such as preconditioned PPA \cite{he2009} and the Bregman PPA \cite{censor1992}. For additional applications of PPA, see \cite{parikh2014}. For a thorough understanding of PPA, we refer to \cite{cai2022}.

In 1983, Nesterov introduced the Nesterov's accelerated gradient (NAG) method \cite{nesterov83}, which exhibits $O(1/k^2)$ rate of convergence and has the same computational complexity as vanilla gradient descent. This significant theoretical result sparked considerable interest in developing accelerated versions of optimization algorithms, as illustrated in \cite{nesterov2005, beck2009}. As a result, the study of accelerated variants of PPA has gained researchers' interest. In \cite{guler1992}, G\"{u}ler employed  the estimation sequence technique, which is firstly used to construct the NAG method in \cite{nesterov83}, to develop an accelerated PPA \cite{guler1992}. Subsequently, he applied the same technique to ALM to derive the accelerated ALM \cite{guler1992augmented}. 

Recently, there has been increased research works towards understanding and analyzing accelerated optimization algorithms from the viewpoint of ordinary differential equations (ODEs). Originally, \cite{su2016} indicated that the NAG method is closely related to a second-order ODE and used such ODE to establish an Lyapunov function framework, which demonstrates the $O(1/k^2)$ convergence rate of NAG method. Following this insight, a number of studies have focused on interpreting accelerated optimization algorithms through the ODE perspective, such as \cite{attouch1996, attouch2016fast, alvarez2002, chen2022, krichene2015, shi2022}. Moreover, there are some works that obtain new accelerated optimization algorithms by applying numerical algorithms for solving ODEs. For example, \cite{zhang2018} developed a new accelerated gradient method by using the Runge-Kutta method on the relevant ODEs. Similarly, \cite{betancourt2018} used the symplectic integrator method on ODEs to obtain a new first-order accelerated method. Additionally, some works propose new ODEs and then use difference schemes to derive new accelerated algorithms, 
as exemplified by \cite{bot2023ogda, bot2023fast, bot2023augmented}.

\subsection{Existing Results}

Let \( \mathcal{H} \) be a real Hilbert space, and let \( \Gamma_0(\mathcal{H}) \) denote the set of all closed, proper, convex functions defined on $\mathcal{H}$. The inner product on \( \mathcal{H} \) is denoted by \( \langle \cdot, \cdot \rangle \), and the corresponding norm is \( \|\cdot\| = \sqrt{\langle \cdot, \cdot \rangle} \). Suppose that \( L: \mathcal{H} \to \mathcal{H} \) is a positive definite operator. The inner product with respect to \( L \) is given by \( \langle x, y \rangle_L = \langle Lx, y \rangle \), and the norm with respect to \( L \) is given by \( \|x\|_L = \sqrt{\langle Lx, x \rangle} \). 

Throughout this paper, we consider the following optimization problem:
\begin{equation}
	\min_{x\in\mathcal{H}} f(x),
	\label{eq:main}
\end{equation}
where \( f \in \Gamma_0(\mathcal{H}) \). Here, we assume that the set of all minimizers of $f$, labeled as $\Omega$, is non-empty. The preconditioned proximal point algorithm (PPA) with respect to the preconditioner \( L \) is given by:
\begin{equation}
	x_{k+1} = \argmin_{x\in\mathcal{H}} \left\{ f(x) + \frac{1}{2\rho_k} \|x - x_k\|^2_L \right\},
	\label{eq:ppa}
\end{equation}
where \( \{\rho_k\} \) is a sequence in \( (0, \infty) \). If \( L \) is the identity operator \( I \), then \eqref{eq:ppa} reduces to the standard PPA.

In \cite{guler1991}, G\"{u}ler proved that the convergence rate of the standard PPA is:
\[
f(x_k) - f^* \leqslant \frac{\dist_I(x_0, \Omega)^2}{2 \sum_{i=0}^{k-1} \rho_i},
\]
where \( f^* \) is the minimum value of \( f \), and \( \dist_L(x_0, \Omega) = \inf_{x \in \Omega} \|x_0 - x\|_L \).

In \cite{guler1992}, G\"{u}ler proposed the following accelerated PPA given as follows:\begin{algorithm}[ht]
	\caption{Accelerated PPA}
	\label{al:appa}
	\textbf{Initialization:} Choose an initial point $x_0$ such that $f(x_0)<\infty$, and constants $\rho_0>0$ and $A>0$. Define $v_0:=x_0, A_0=A$\;
	
	\For{$k=0, 1, \cdots$}{
		Choose $\rho_k>0$\;
		$\alpha_k=\dfrac{\sqrt{(A_k\rho_k)^2+4A_k\rho_k}-A_k\rho_k}{2}$\;
		$y_k=(1-\alpha_k)x_k+\alpha_kv_k$\;
		$x_{k+1}=\argmin_x\left\{f(x)+\dfrac{1}{2\rho_k}\norm{x-y_k}^2\right\}$\;
		$v_{k+1}=v_k+\dfrac{1}{\alpha_k}(x_{k+1}-y_k)$\;
		$A_{k+1}=(1-\alpha_k)A_k$\;
	}
\end{algorithm}

Also, G\"{u}ler proved that the convergence rate of Algorithm \ref{al:appa} is 
\begin{equation}
	\label{eq:apparate}
	f(x_k)-f^*\leqslant\frac{4}{A(\sum_{i=0}^{k-1}\sqrt{\rho_i})^2}\left(f(x_0)-f^*+\frac{A}{2}\dist_I(x_0,\Omega)^2\right).
\end{equation}

\subsection{Our Contributions}
In this paper, we aim to propose a new accelerated preconditioned PPA by employing the Symplectic Euler Method to discretize a first-order ODE. Owing to the usage of Symplectic Euler Method, we call this new PPA as \emph{Symplectic Proximal Point Algorithm (SPPA)}. The details of the development and theory of SPPA are provided as follows.

First, inspired by the ODE used to analyze the accelerated mirror descent method in \cite{yi23}, we propose the following ODE:
\begin{equation}
	\label{eq:continuousmain}
	\begin{split}
		Z &= b_t \dot{X} + c_t L^{-1} \nabla f(X) + X, \\
		\dot{Z} &= -a_t L^{-1} \nabla f(X), \\
		Z(0) &= X(0) = x_0,
	\end{split}
\end{equation}
where \( f \in \Gamma_0(\mathcal{H}) \) is continuously differentiable, and \( a_t, b_t, \) and \( c_t \) are functions defined on \([0, +\infty)\) such that they are positive on \((0, +\infty)\) and non-negative at \( 0 \). Also, we assume that $a_tb_t$ is differentiable respect to $t$. By applying the Lyapunov function technique, we first prove that the convergence rate of the solution trajectory to \eqref{eq:continuousmain} is \( O(1/a_t b_t) \) under the assumption that \( 0\leqslant\dfrac{d}{dt}(a_t b_t) \leqslant a_t \). Additionally, we demonstrate that the convergence rate of the solution trajectory to \eqref{eq:continuousmain} is \( o(1/a_t b_t) \) under some additional assumptions.

Next, we apply the Symplectic Euler Method to \eqref{eq:continuousmain} to obtain our SPPA. The motivation for applying the Symplectic Euler Method comes from the works in \cite{shi2019, shi2022}. In these works, it was shown that the NAG method can be viewed as applying the symplectic method to the high-resolution ODEs formulated in phase space representation. Since the Symplectic Euler Method can preserve geometric structure of the ODEs and has a simple iteration rule, we use the Symplectic Euler Method to derive our algorithm. 

First, we give a concise introduction to the Symplectic Euler Method. Consider the following Hamiltonian system:

Consider the following Hamiltonian system:\begin{equation}
	\begin{split}
		\dot{p} &= -\frac{\partial H(p,q)}{\partial q},\\
		\dot{q} &= \frac{\partial H(p,q)}{\partial p}.
	\end{split}
	\label{eq:hamitonian}
\end{equation}
The Symplectic Euler Method  for solving \eqref{eq:hamitonian} is given as follows:
\begin{equation}
	\begin{split}
		p_{k+1} &= p_k - s\frac{\partial H(p_k,q_{k+1})}{\partial q},\\
		q_{k+1} &= q_k + s\frac{\partial H(p_k,q_{k+1})}{\partial p},
	\end{split}\quad\text{or}\quad\begin{split}
		p_{k+1} &= p_k - s\frac{\partial H(p_{k+1},q_k)}{\partial q},\\
		q_{k+1} &= q_k + s\frac{\partial H(p_{k+1},q_k)}{\partial p},
	\end{split}
	\label{eq:symplecticeuler}
\end{equation}
where $s>0$ is the step-size of Symplectic Euler Method. Because one variable takes an explicit step and the other variable takes an implicit step, this method is also referred to as the Semi-Implicit Euler Method. To acquire more acknowledge about symplectic methods, we refer to \cite{hairer06} and \cite{feng10}. The introduction of Symplectic Euler Method can be found in I.1.2 in \cite{hairer06}.
\begin{figure}[ht]
	\centering
	\subfigure{\includegraphics[width=.3\textwidth]{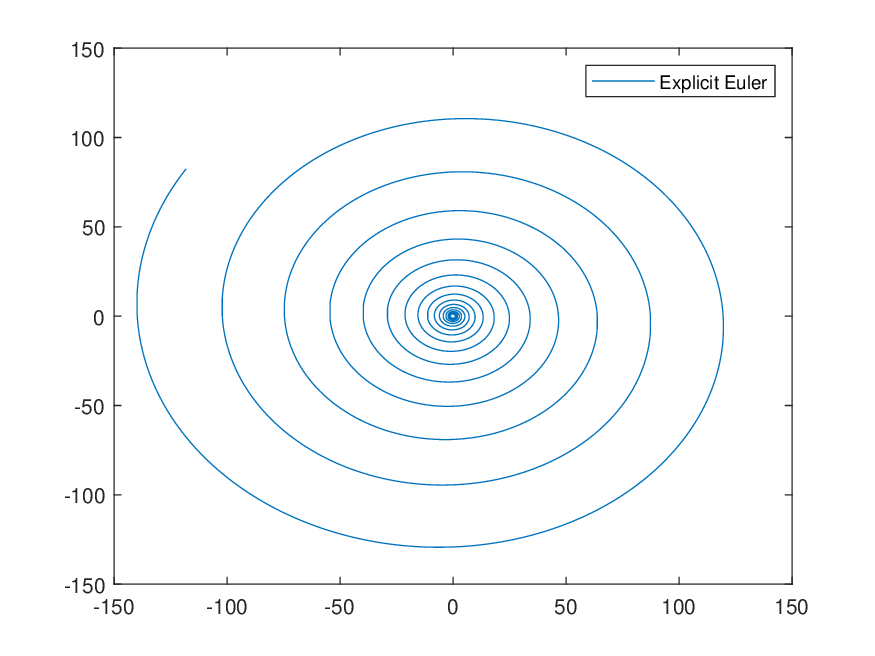}}
	\subfigure{\includegraphics[width=.3\textwidth]{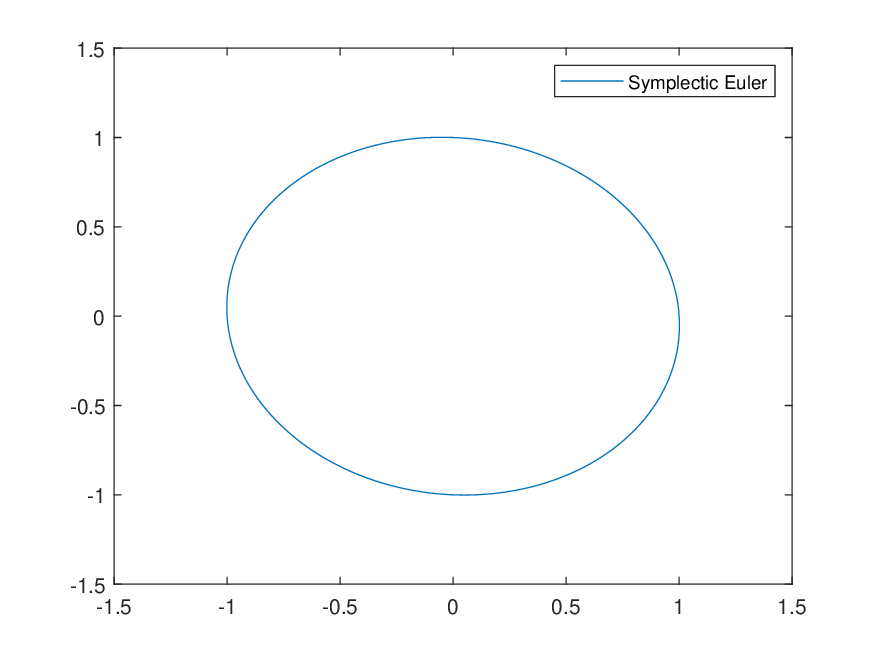}}
	\subfigure{\includegraphics[width=.3\textwidth]{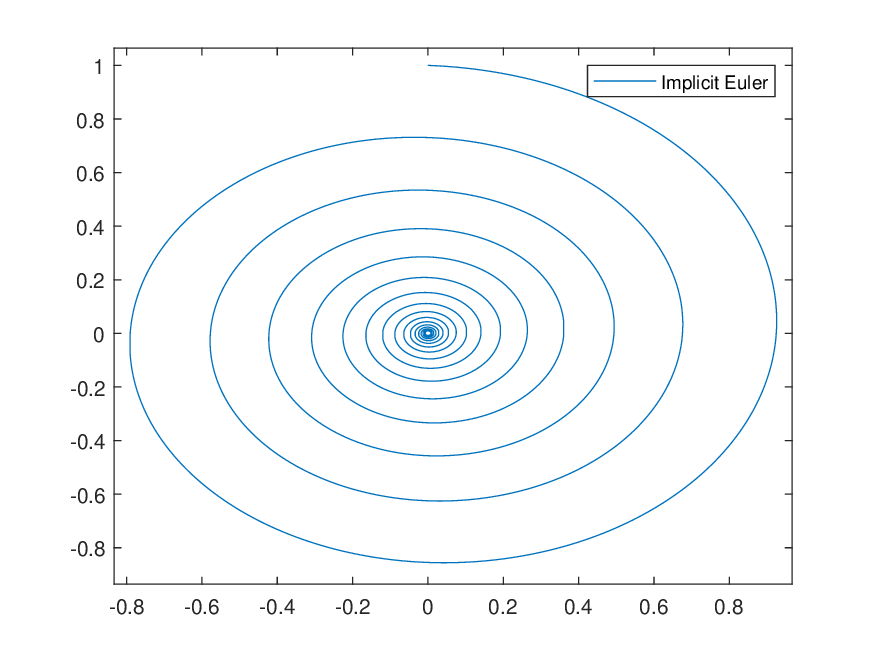}}
	\caption{Comparison between some existing Euler methods. Here we apply the Explicit Euler Method, Symplectic Euler Method and Implicit Euler Method to the Hamiltonian system $\dot{p}=q$, $\dot{q}=-p$ with initial $p(0)=0$ and $q(0)=1$. The solution to such Hamiltonian system is $p=\sin t$, $q=\cos t$. As we can see from these three figures, the Symplectic Euler method preserve the structure of the solution.}
\end{figure}

Here we apply the Symplectic Euler Method to discretize \eqref{eq:continuousmain}, where $Z$ takes the explicit step, $X$ takes the implicit step, and $s=1$. Then, we obtain the following recursive rule:
\begin{equation}
	\label{eq:discretemain}
	\begin{split}
		z_0&=x_0,\\
		z_k&=b_k(x_{k+1}-x_k)+c_kL^{-1}\nabla f(x_{k+1})+x_{k+1},\\
		z_{k+1}&=z_k-a_kL^{-1}\nabla f(x_{k+1}),
	\end{split}
\end{equation}
where the sequences $\{a_k\}$, $\{b_k\}$ and $\{c_k\}$ are positive if $k\geqslant 1$ and non-negative if $k=0$. However, the convergence results of \eqref{eq:ppa} do not require $f$ to be differentiable. Therefore, we need to eliminate the usage of $\nabla f$ from \eqref{eq:discretemain}. By using the first-order characterization of the convex optimization problem, the update rule of $x_{k+1}$ can be transformed into:\begin{align*}
	\tilde{x}_{k+1} &= \frac{1}{b_k+1}z_k+\frac{b_k}{b_k+1}x_k, \\
	x_{k+1} & = \argmin_x\left\{f(x)+\frac{b_k+1}{2c_k}\norm{x-\tilde{x}_{k+1}}^2_L\right\}.
\end{align*}
Additionally, using the relation $L^{-1}\nabla f(x_{k+1})=c_k^{-1}[b_kx_k+z_k-(b_k+1)x_{k+1}]$, the update rule for $z_{k+1}$ can be rewritten as:\[
z_{k+1}=z_k+\frac{a_k(b_k+1)}{c_k}(x_{k+1}-\tilde{x}_{k+1}).
\]
In conclusion, we obtain the following algorithm.
\begin{algorithm}
	\caption{Symplectic Proximal Point Algorithm}
	\label{al:spppa}
	\For{$k=0, 1, \cdots$}{
		\ Choose $a_k$, $b_k$ and $c_k$ such that all of them are positive if $k\geqslant 1$ and non-negative for $k=0$\;
		\ $\tilde{x}_{k+1} = \dfrac{1}{b_k+1}z_k+\dfrac{b_k}{b_k+1}x_k$\;
		\ $x_{k+1} = \argmin\limits_{x\in\mathcal{H}}\left\{f(x)+\dfrac{b_k+1}{2c_k}\norm{x-\tilde{x}_{k+1}}^2_L\right\}$\;
		\ $z_{k+1}=z_k+\dfrac{a_k(b_k+1)}{c_k}(x_{k+1}-\tilde{x}_{k+1})$\;
	}
\end{algorithm}

Theoretically, by applying the Lyapunov function technique, we prove that the convergence rate of Algorithm \ref{al:spppa} is \( O(1/a_k b_k) \) under the assumptions that \( 0 \leqslant a_{k+1} b_{k+1} - a_k b_k \leqslant a_k \) and \( c_k \geqslant \dfrac{c_k}{2} \). Under these conditions, we demonstrate that SPPA can be seen as a generalization of A-PPA. Moreover, the corresponding convergence rate of SPPA is shown to be faster than that of A-PPA. This enhanced convergence rate may be attributed to the structure-preserving property of the Symplectic Euler Method. Additionally, under some extra assumptions, we prove that the convergence rate of Algorithm \ref{al:spppa} is \( o(1/a_k b_k) \).

\section{Convergence Rate of ODE}
\label{sec:2}

In this section, we study the convergence rate of \eqref{eq:continuousmain}. For simplicity, when proving the theorems in this section, we use the notations \( X \), \( Z \), \( \X \), and \( \Z \) to denote \( X(t) \), \( Z(t) \), \( \X(t) \), and \( \Z(t) \) respectively.

First, we need to establish the \( O(1/a_t b_t) \) convergence rate. To achieve this, we propose the following Lyapunov function:
\begin{equation}
	\E(t) := A_t [f(X(t)) - f(x^*)] + \frac{1}{2} \|Z(t) - x^*\|_L^2,
	\label{eq:lyapunovcontinuous}
\end{equation}
where \( x^* \in \Omega \). Here, \( A_t [f(X(t)) - f(x^*)] \) represents the potential term of \( \E(t) \), and \( A_t \) indicates the convergence rate of the solution trajectory to \eqref{eq:continuousmain}. \( \dfrac{1}{2} \|Z(t) - x^*\|_L^2 \) is the mixed term. By analyzing \eqref{eq:lyapunovcontinuous}, we obtain the following result:

\begin{theorem}
	\label{thm:continuousrate}
	Let $(X(t),Z(t))$ be the solution to \eqref{eq:continuousmain}. If  \begin{equation}
		A_t=a_tb_t,\qquad 0\leqslant\dot{A}_t\leqslant a_t,
		\label{eq:continuouscondition}
	\end{equation}
	and $f\in\Gamma_0(\hh)$ is continuous differentiable, then $\E(t)$ is non-increasing. And we have
	\begin{equation}
		f(X(t))-f^* \leqslant \frac{A_0}{A_t}[f(x_0)-f^*]+\frac{\dist_L(x_0,\Omega)^2}{2A_t}.
		\label{eq:continuousrate}
	\end{equation}
	Additionally, we have
	\begin{gather}
		\int_{0}^{\infty}(a_t-\dot{A}_t)\inner{\nabla f(X(t)),X(t)-x^*}\dd t\leqslant A_0[f(x_0)-f^*]+\frac{1}{2}\norm{x_0-x^*}^2_L,\ \forall x^*\in\Omega,\label{eq:integralvalue}\\
		\int_{0}^{\infty}a_tc_t\norm{\nabla f(X(t))}^2\dd t\leqslant A_0[f(x_0)-f^*]+\frac{1}{2}\dist_L(x_0,\Omega)^2.\label{eq:integralnorm}
	\end{gather}

\end{theorem}
\begin{proof}
	Step 1: Estimating the upper bound of $\EE(t)$. By calculating the derivative of \eqref{eq:lyapunovcontinuous}, we have\begin{align*}
		&\ \EE(t)=\dot{A}_t[f(X)-f(x^*)]+A_t\inner{\nabla f(X),\X}+\inner{\Z,Z-x^*}_L \\
		=&\ \dot{A}_t[f(X)-f(x^*)]+A_t\inner{\nabla f(X),\X}-a_t\inner{\nabla f(X),b_t\X+c_tL^{-1}\nabla f(X)+X-x^*}\\
		=&\  \dot{A}_t[f(X)-f(x^*)]-a_t\inner{\nabla f(X),X-x^*}-a_tc_t\norm{\nabla f(X)}^2_{L^{-1}}.
	\end{align*}
	Since $f$ is convex, we have\[
	f(X)-f(x^*)\leqslant\inner{\nabla f(X),X-x^*}.
	\]
	Thus\[
	\EE(t)\leqslant (\dot{A}_t-a_t)\inner{\nabla f(X(t)),X(t)-x^*}-a_tc_t\norm{\nabla f(X)}_{L^{-1}}^2\leqslant 0.
	\]
	
	Step 2: Deducing the convergence rates. Because $\E(t)$ is non-increasing, we have\[
	A_t[f(X)-f(x^*)]\leqslant\E(t)\leqslant\E(0)=A_0[f(x_0)-f(x^*)]+\frac{1}{2}\norm{x_0-x^*}_L^2.
	\]
	By dividing $A_t$ on both sides of the above inequality and  taking the infimum over to all $x^*\in\Omega$, we obtain \eqref{eq:continuousrate}.
	
	Next, by integrating $\EE(t)$ from $0$ to $\infty$, we have\begin{gather}
		\int_{0}^{\infty}(a_t-\dot{A}_t)\inner{\nabla f(X(t)),X(t)-x^*}\dd t  \leqslant-\int_{0}^{\infty}\EE(t)\dd t\leqslant \E(0), \label{eq:gather1}\\
		\int_{0}^\infty a_tc_t\norm{\nabla f(X(t))}_{L^{-1}}^2\dd t \leqslant -\int_{0}^{\infty}\EE(t)\dd t\leqslant \E(0).
	\end{gather}
	It is obvious that \eqref{eq:gather1} is exactly \eqref{eq:integralvalue}. Also, By taking the infimum over all $x^*\in\Omega$, we obtain \eqref{eq:integralnorm}.
\end{proof}

Next, we introduce some feasible choices for $a_t$, $b_t$, and $c_t$.

\begin{example}
	\label{ex:continuouspolynomial}
	First, we present an instance of \eqref{eq:continuousmain} with an $O(1/t^p)$ convergence rate. Let \( A_t = t^p \), where \( p \geqslant 1 \). Then \( a_t = \dot{A}_t = pt^{p-1} \), \( b_t = tp^{-1} \), and \( c_t = a_t \). These choices satisfy the assumptions in Theorem \ref{thm:continuousrate}. By Theorem \ref{thm:continuousrate}, the solution trajectory to \eqref{eq:continuousmain} is \( O(1/t^p) \).
\end{example}

\begin{example}
	\label{ex:continuousexponential}
	Next, we derive an instance of \eqref{eq:continuousmain} with an $O(\exp(-\lambda t))$ convergence rate. Let \( A_t = \exp(\lambda t) \), where \( \lambda > 0 \). Then \( a_t = \dot{A}_t = \lambda \exp(\lambda t) \), \( b_t = \lambda^{-1} \), and \( c_t = a_t \). These choices satisfy the assumptions in Theorem \ref{thm:continuousrate}. By Theorem \ref{thm:continuousrate}, the solution trajectory to \eqref{eq:continuousmain} is \( O(\exp(-\lambda t)) \).
\end{example}

In the remainder of this section, we discuss the conditions under which the solution trajectory to \eqref{eq:continuousmain} achieves a convergence rate of \( o(1/A_t) \). Inspired by the proof of the \( o(1/k^2) \) convergence rate of the NAG method, we find that \eqref{eq:integralvalue} is critical for proving the \( o(1/A_t) \) convergence rate of \eqref{eq:continuousmain}. To make \eqref{eq:integralvalue} meaningful, we introduce the following assumption:

\begin{assumption}
	\label{ass:continuous1}
	There exists a constant \( d \in (0, 1) \) such that \( 0 \leqslant \dot{A}_t \leqslant d a_t \).
\end{assumption}

Let us revisit Example \ref{ex:continuousexponential}. In this case, \( a_t \) is of the same order as \( A_t \). Thus, if Assumption \ref{ass:continuous1} holds, then by \eqref{eq:integralvalue}, we can easily prove the \( o(1/A_t) \) convergence rate of \eqref{eq:continuousmain}. The resulting theorem is given as follows:

\begin{theorem}
	\label{thm:oecontinuousrate}
	Let \((X(t), Z(t))\) be the solution to \eqref{eq:continuousmain}. If the following conditions hold:
	\begin{itemize}
		\item \( A_t = a_t b_t \) and Assumption \ref{ass:continuous1} holds,
		\item there exists a positive constant $\gamma$ and a non-negative constant $t_0$ such that \( a_t \geqslant \gamma A_t \) for all \( t \geqslant t_0 \),
	\end{itemize}
	then
	\[
	\lim_{t \to \infty} A_t [f(X(t)) - f^*] = 0.
	\]
\end{theorem}

\begin{proof}
	By Assumption \ref{ass:continuous1}, we have that\[
	\int_{0}^{\infty}(1-d)a_t[f(X(t))-f^*]\dd t\leqslant\int_{0}^{\infty}(1-d)a_t\inner{\nabla f(X(t)),X(t)-x^*}\dd t<\infty.
	\]
	Then we have\[
	\lim_{t\to\infty}a_t[f(X(t))-f^*]=0.
	\]
	Because $a_t\geqslant \gamma A_t$ for all $t\geqslant t_0$, we can obtain\[
	\lim_{t\to\infty}A_t[f(X(t))-f^*]=0.
	\]
\end{proof}

However, if \( a_t \) is an infinitesimal of higher order than \( A_t \), such as the one demonstrated in Example \ref{ex:continuouspolynomial}, we need some extra assumptions and inductions to obtain the \( o(1/A_t) \) convergence rate. First, we introduce the following useful assumption and lemma:

\begin{assumption}
	\label{ass:continuous2}
	\( c_t \) is differentiable and \( \sup\limits_{t > 0} \dfrac{\dot{c}_t}{a_t} < \infty \).
\end{assumption}

\begin{lemma}
	\label{lem:continuouso}
	Let \((X(t), Z(t))\) be the solution to \eqref{eq:continuousmain}, and let \(\G(t)\) be an auxiliary function defined as
	\begin{equation}
		\G(t) := c_t [f(X(t)) - f(x^*)] + \frac{1}{2} \|X(t) - x^*\|^2_L - \langle X(t) - x^*, Z(t) - x^* \rangle_L.
	\end{equation}
	If Assumptions \ref{ass:continuous1} and Assumption \ref{ass:continuous2} hold, and \( f \in \Gamma_0(\mathcal{H}) \) is continuously differentiable, then we have
	\[
	\E_\alpha := \E + \alpha \G \text{ is non-increasing and non-negative}, \  \forall \alpha \in \left(0, \frac{1 - d}{1 + \sup_{t > 0} \frac{\dot{c}_t}{a_t}}\right).
	\]
	Additionally, we have that $\norm{X(t)-x}^2_L$ is uniformly bounded, and
	\begin{equation}
		\int_{0}^\infty b_t \|\X(t)\|^2_L \, dt < \infty.
		\label{eq:integralvelocity}
	\end{equation}
\end{lemma}

\begin{proof}
	The derivative of $\G(t)$ can be calculated as follows:
	\begin{align*}
		\dot{\G}(t)= &\ \dot{c}_t[f(X)-f(x^*)]+c_t\inner{\nabla f(X),\X}+\inner{\X,X-x^*}_L\\
		&\ -\inner{\X,Z-x^*}_L-\inner{X-x^*,\Z}_L\\
		= &\ -b_t\|\X\|^2_L+a_t\inner{\nabla f(X),X-x^*}+\dot{c}_t[f(X)-f(x^*)].
	\end{align*}
	
	By previous calculations and the proof for Theorem \ref{thm:continuousrate}, the upper bound of $\dot{\E}_\alpha$ can be estimated as follows:\begin{align*}
		\EE_\alpha(t)= & \ (\dot{A}_t+\alpha\dot{c}_t)[f(X)-f(x^*)]-(1-\alpha)a_t\inner{\nabla f(X),X-x^*}\\
		&\ -\alpha b_t\|\X\|^2_L-a_tc_t\norm{\nabla f(X)}^2_{L^{-1}}\\
		\leqslant&\ [-(1-\alpha-d)a_t+\alpha\dot{c}_t][f(X)-f(x^*)]-\alpha b_t\|\X\|^2\\
		\leqslant&\ 0.
	\end{align*}
	Also, since $\alpha<1$, $\E_\alpha(t)$ can be represented as sum of non-negative terms as follows:
	\begin{align*}
		\E_\alpha(t)=&\ (A_t+\alpha c_t)[f(X)-f(x^*)]+\frac{\alpha-\alpha^2}{2}\norm{X-x^*}^2_L+\frac{1}{2}\norm{Z-x^*-\alpha(X-x^*)}^2.
	\end{align*}
	Then we have\[
	\norm{X-x^*}^2\leqslant\frac{2\E_\alpha(0)}{\alpha-\alpha^2}<\infty,
	\]
	which shows that $\norm{X(t)-x^*}^2_L$ is uniformly bounded.
	
	Next, for all $\alpha\in\left(0,\frac{1 - d}{1 + \sup_{t > 0} \frac{\dot{c}_t}{a_t}}\right)$, we integral $\dot{\E}_\alpha(t)$ from $0$ to $\infty$, and obtain\[
	\int_{0}^{\infty}\alpha b_t\|\X\|_L\dd t\leqslant\E_\alpha(0)<\infty.
	\]
	Thus, we obtain \eqref{eq:integralvelocity}.
\end{proof}
With the above lemma, we can obtain the following theorem.

\begin{theorem}
	Let $(X(t),Z(t))$ be the solution to \eqref{eq:continuousmain}. If \begin{itemize}
		\item $A_t=a_tb_t$, Assumption \ref{ass:continuous1} and Assumption \ref{ass:continuous2} hold,
		\item there exists a positive constant $\gamma$ and a non-negative constant $t_0$ such that $c_t\leqslant \gamma a_t$ for all $t\geqslant t_0$,
		\item and $\beta:=\lim\limits_{t\to\infty}\dfrac{t}{b_t}\in(0,\infty)$,
	\end{itemize}   then we have \[
	\lim_{t\to\infty}A_t[f(X(t))-f^*]=0.
	\]
\end{theorem}
\begin{proof}
	By Lemma \ref{lem:continuouso} and Theorem \ref{thm:continuousrate}, we have that $\lim\limits_{t\to\infty}\G(t)$ exists. Since $f(X(t))-f(x^*)\leqslant O(1/A_t)$ and $c_t\leqslant \gamma a_t$ for all $t\geqslant t_0$, then we have $\lim\limits_{t\to\infty} c_t[f(X(t))-f(x^*)]=0$. Thus\[
	\lim_{t\to\infty}\frac{1}{2}\norm{X(t)-x^*}_L^2-\inner{X(t)-x^*,Z(t)-x^*}_L\text{ exists.}
	\]
	Next, we will show that $\lim\limits_{t\to\infty}\dfrac{1}{2}\norm{X(t)-x^*}_L^2$ exists. By calculating the derivative of $\dfrac{1}{2}\norm{X-x^*}_L^2$, we have
	\begin{align*}
		\frac{\dd}{\dd t}\frac{1}{2}\norm{X-x^*}^2_L=&\ \inner{\X,X-x^*}_L\\
		=&\ \frac{1}{b_t}\inner{Z-X-c_tL^{-1}\nabla f(X),X-x^*}_L\\
		=&\ -\frac{1}{b_t}\norm{X-x^*}^2_L+\frac{1}{b_t}\inner{X-x^*,Z-x^*}_L-\frac{c_t}{b_t}\inner{\nabla f(X),X-x^*}
	\end{align*}
	Then we have\begin{align*}
		&\ \frac{1}{2}\norm{X-x^*}^2_L+\frac{t}{\beta}\frac{\dd}{\dd t}\frac{1}{2}\norm{X-x^*}^2_L \\
		= &\ \frac{1}{2}\norm{X-x^*}^2_L-\frac{t}{b_t\beta}\norm{X-x^*}^2_L+\frac{t}{b_t\beta}\inner{X-x^*,Z-x^*}_L-\frac{c_tt}{b_t\beta}\inner{\nabla f(X),X-x^*}.
	\end{align*}
	Since $c_t\leqslant \gamma a_t$ for all $t\geqslant t_0$, then by using \eqref{eq:integralvalue}, we have that\[
	\lim_{t\to\infty}\frac{c_tt}{b_t\beta}\inner{\nabla f(X),X-x^*}=0.
	\]
	Next, because $\lim\limits_{t\to\infty}\dfrac{1}{2}\norm{X-x^*}^2_L-\inner{X-x^*,Z-x^*}_L$ exists and $\norm{X-x^*}_L^2$ is uniformly bounded, we have that
	\begin{align*}
		&\ \lim_{t\to\infty}\left(\frac{1}{2}-\frac{t}{b_t\beta}\right)\norm{X-x^*}^2_L+\frac{t}{b_t\beta}\inner{X-x^*,Z-x^*}_L\\
		=& -\lim_{t\to\infty}\dfrac{1}{2}\norm{X-x^*}^2_L-\inner{X-x^*,Z-x^*}_L.
	\end{align*}
	Thus, $\lim\limits_{t\to\infty}\dfrac{1}{2}\norm{X-x^*}_L^2+\dfrac{t}{\beta}\dfrac{\dd}{\dd t}\dfrac{1}{2}\norm{X-x^*}_L^2$ exists. Then by Lemma \ref{lem:tlimitexist}, $\lim\limits_{t\to\infty}\dfrac{1}{2}\norm{X-x^*}_L^2$ exists. Also, we have\[
	\lim\limits_{t\to\infty}\inner{b_t\X+c_tL^{-1}\nabla f(X),X-x^*}_L \text{ exists.}
	\]
	
	Finally, we can prove $\lim\limits_{t\to\infty}A_t[f(X)-f^*]=0$. Owing to the existence of $\lim\limits_{t\to\infty}\E(t)$, $\lim\limits_{t\to\infty}\dfrac{1}{2}\norm{X-x^*}^2_L$, and $\lim\limits_{t\to\infty}\inner{b_t\X+c_tL^{-1}\nabla f(X),X-x^*}_L$, we have that\[
	\lim_{t\to\infty}A_t[f(X)-f(x^*)]+\frac{1}{2}\norm{b_t\X+c_tL^{-1}\nabla f(X)}^2_L\text{ exists.}
	\]
	By Cauchy-Schwarz inequality, we have\[
	\frac{1}{2}\norm{b_t\X+c_tL^{-1}\nabla f(X)}^2_L\leqslant b_t^2\|\X\|^2_L+c_t^2\norm{\nabla f(X)}^2_{L^{-1}}.
	\]
	Then by \eqref{eq:integralvalue}, \eqref{eq:integralnorm}, \eqref{eq:integralvelocity}, and the assumptions that $\lim\limits_{t\to\infty}\dfrac{t}{b_t}\in(0,\infty)$  and $c_t\leqslant \gamma a_t$, we have that \[
	\lim\limits_{t\to\infty}A_t[f(X)-f(x^*)]=\lim_{t\to\infty}A_t[f(X)-f(x^*)]+\frac{1}{2}\norm{b_t\X+c_tL^{-1}\nabla f(X)}^2_L=0.
	\]
\end{proof}

To demonstrate the $o(1/A_t)$ convergence rate of \eqref{eq:continuousmain}, we have introduced several additional assumptions. However, it can be shown that these assumptions are valid for some typical scenarios.

\begin{example}
	Firstly, let us revisit the parameters introduced in Example \ref{ex:continuouspolynomial}. Given that $A_t = a_t b_t$ and $\dot{A}_t = a_t$, we introduce a constant $d \in (0,1)$ and set $\tilde{a}_t = d^{-1} a_t = d^{-1} p t^{p-1}$, $\tilde{b}_t = d \tilde{b}_t = d p^{-1} t$. Consequently, $A_t = \tilde{a}_t\tilde{b}_t$ and $\dot{A}_t \leqslant d \tilde{a}_t$. By setting $\tilde{c}_t = \tilde{a}_t$, the condition that $\tilde{c}_t \leqslant \gamma \tilde{a}_t$ is met. Furthermore, $\lim\limits_{t\to\infty}tb_t^{-1}\in(0,\infty)$ holds true. Hence, the $o(1/t^p)$ convergence rate of \eqref{eq:continuousmain} can be established.
\end{example}

\begin{example}
	Secondly, we examine the parameters presented in Example \ref{ex:continuousexponential}. Given two constants $\lambda \in (1, \infty)$ and $d \in (0,1)$, we define $A_t = \exp(\lambda t)$, $a_t = d^{-1} \lambda \exp(\lambda t)$, $b_t = d \lambda^{-1}$, and $c_t = a_t$. Under these definitions, the assumptions given in Theorem \ref{thm:oecontinuousrate} satisfied, thereby implying the $o(\exp(-\lambda t))$ convergence rate for \eqref{eq:continuousmain}.
\end{example}

\section{Convergence Rate of SPPA}

In this section, we establish the convergence rate of Algorithm \ref{al:spppa}. For notational convenience, we introduce the term $\tilde{\nabla}f(x_{k+1}) := \dfrac{(b_k + 1)}{c_k} L (\tilde{x}_{k+1} - x_{k+1})$. By the first-order characterization of $x_{k+1}$, it follows that $\tilde{\nabla}f(x_{k+1}) \in \partial f(x_{k+1})$. Analogous to the discussion in Section \ref{sec:2}, our initial step is to demonstrate both the last-iterate and the convergence rate in a summation form of Algorithm \ref{al:spppa}. To achieve this, we define the following discrete-time Lyapunov function:

\begin{equation}
	\label{eq:lyapunovdiscrete}
	\mathcal{E}(k) := A_k [f(x_k) - f(x^*)] + \frac{1}{2} \|z_k - x^*\|_L^2.
\end{equation}
This Lyapunov function serves as a discrete counterpart to \eqref{eq:lyapunovcontinuous}. Through the analysis of \eqref{eq:lyapunovdiscrete}, we derive the subsequent theorem.

\begin{theorem}
	\label{thm:discreterate}
	Let $\{x_k\}$, $\{\tilde{x}_{k+1}\}$, and $\{z_k\}$ denote the sequences generated by Algorithm \ref{al:spppa}. Assuming the conditions
	\begin{equation}
		\label{eq:discretecondition}
		A_k = a_k b_k, \qquad 0 \leqslant A_{k+1} - A_k \leqslant a_k, \qquad c_k \geqslant \frac{a_k}{2},
	\end{equation}
	and $f \in \Gamma_0(\mathcal{H})$, the sequence $\{\mathcal{E}(k)\}$ is non-increasing. Consequently, the last-iterate convergence rate of Algorithm \ref{al:spppa} is given by
	\begin{equation}
		f(x_k) - f^* \leqslant \frac{A_0}{A_k} [f(x_0) - f^*] + \frac{\mathrm{dist}_L(x_0, \Omega)^2}{2 A_k}.
		\label{eq:discreterate}
	\end{equation}
	Additionally, for any $x^* \in \Omega$,
	\begin{gather}
		\sum_{k=0}^{\infty} (a_k + A_k - A_{k+1}) \langle \tilde{\nabla}f(x_{k+1}), x_{k+1} - x^* \rangle \leqslant A_0 [f(x_0) - f^*] + \frac{1}{2}\norm{x_0-x^*}^2_L,\ \forall x^*\in\Omega, \label{eq:valuesum} \\
		\sum_{k=0}^{\infty} \left(a_k c_k - \frac{a_k^2}{2}\right) \| \tilde{\nabla}f(x_{k+1}) \|^2 \leqslant A_0 [f(x_0) - f^*] + \frac{1}{2} \mathrm{dist}_L(x_0, \Omega)^2. \label{eq:normsum}
	\end{gather}
\end{theorem}

\begin{proof}
	
	Step 1: Estimating the upper bound of $\E(k+1)-\E(k)$. First, we can divide $\E(k+1)-\E(k)$ into three parts as follows:\begin{align*}
		\E(k+1)-\E(k)= &\ A_k[f(x_{k+1})-f(x_k)]+(A_{k+1}-A_k)[f(x_{k+1})-f(x^*)]\\
		&\ +\frac{1}{2}\norm{z_{k+1}-x^*}^2_L-\frac{1}{2}\norm{z_k-x^*}^2_L.
	\end{align*}
	Next, we estimate the upper bound of first three terms. Since $\tilde{\nabla}f(x_{k+1})\in\partial f(x_{k+1})$, we have\begin{align*}
		A_k[f(x_{k+1})-f(x_k)] &\leqslant A_k\inner{\tilde{\nabla}f(x_{k+1}),x_{k+1}-x_k},  \\
		(A_{k+1}-A_k)[f(x_{k+1})-f(x^*)] & \leqslant (A_{k+1}-A_k)\inner{\tilde{\nabla}f(x_{k+1}),x_{k+1}-x^*}.
	\end{align*}
	Also, we have\begin{align*}
		\frac{1}{2}\norm{z_{k+1}-x^*}^2_L-\frac{1}{2}\norm{z_k-x^*}^2_L= &\ \inner{z_{k+1}-z_k,z_k-x^*}_L+\frac{1}{2}\norm{z_{k+1}-z_k}^2_L\\
		= &\ -a_k\inner{\tilde{\nabla}f(x_{k+1}),z_k-x^*}+\frac{a_k^2}{2}\norm{\tilde{\nabla}f(x_{k+1})}^2_{L^{-1}}. 
	\end{align*}
	By summing previous inequalities and using the assumption that $A_k=a_kb_k$, we have\begin{align*}
		\E(k+1)-\E(k)\leqslant &\ a_k\inner{\tilde{\nabla}f(x_{k+1}),(b_k+1)x_{k+1}-b_kx_k-z_k}+\frac{a_k^2}{2}\norm{\tilde{\nabla} f(x_{k+1})}^2_{L^{-1}}\\
		&\ +(A_{k+1}-A_k-a_k)\inner{\tilde{\nabla}f(x_{k+1}),x_{k+1}-x^*}\\
		= &\ \left(-a_kc_k+\frac{a_k^2}{2}\right)\norm{\tilde{\nabla}f(x_{k+1})}^2_{L^{-1}}\\
		&\ +(A_{k+1}-A_k-a_k)\inner{\tilde{\nabla}f(x_{k+1}),x_{k+1}-x^*}.
	\end{align*}
	
	Step 2: Deriving the convergence rates of Algorithm \ref{al:spppa}. Since $c_k\geqslant\dfrac{a_k}{2}$ and $A_{k+1}-A_k-a_k\geqslant 0$, we have $\E(k+1)-\E(k)\leqslant 0$. Then we can obtain the following inequality\[
	A_k[f(x_k)-f(x^*)]\leqslant\E(k)\leqslant\E(0).
	\]
	By dividing $A_k$ on both sides of the above inequality and taking infimum respect to all $x^*\in\Omega$, we obtain \eqref{eq:discreterate}.
	
	Next, by summing $\E(k+1)-\E(k)$ from 0 to $\infty$, we have\begin{gather}
		\sum_{k=0}^{\infty}(a_k+A_k-A_{k+1})\inner{\tilde{\nabla}f(x_{k+1}),x_{k+1}-x^*}  \leqslant\sum_{k=0}^{\infty}\E(k+1)-\E(k)\leqslant\E(0),\label{eq:proof1}\\
		\sum_{k=0}^{\infty}\left(a_kc_k-\frac{a_k^2}{2}\right)\norm{\tilde{\nabla}f(x_{k+1})}^2_{L^{-1}} \leqslant \sum_{k=0}^{\infty}\E(k)-\E(k+1)\leqslant\E(0).\label{eq:proof2}
	\end{gather}
	\eqref{eq:proof1} is \eqref{eq:valuesum}. Next, by taking the infimum over all $x^*\in\Omega$ on both sides of \eqref{eq:proof2}, we obtain \eqref{eq:normsum}.
\end{proof}
\begin{example}
	Here, we show that the conditions given in Theorem \ref{thm:discreterate} are sufficiently weak enough to allow A-PPA to be regarded as a special case of Algorithm \ref{al:spppa}. Let $\{\rho_k\}$ be an arbitrary positive sequence. We construct an instance of Algorithm \ref{al:spppa} such that $\dfrac{c_k}{b_k + 1} = \rho_k$. Inspired by the convergence rate of Algorithm \ref{al:appa}, we set
	\begin{align*}
		A_k &= \frac{1}{2} \left( \sum_{i=0}^{k-1} \sqrt{\rho_i} \right)^2, \\
		a_k &= \frac{\sqrt{\rho_k}}{2} \left( \sqrt{\rho_k} + 2 \sum_{i=0}^{k-1} \sqrt{\rho_i} \right), \\
		b_k &= \frac{\left( \sum_{i=0}^{k-1} \sqrt{\rho_i} \right)^2}{\sqrt{\rho_k} \left( \sqrt{\rho_k} + 2 \sum_{i=0}^{k-1} \sqrt{\rho_i} \right)}, \\
		c_k &= \frac{\sqrt{\rho_k} \left( \sum_{i=0}^{k} \sqrt{\rho_i} \right)^2}{\sqrt{\rho_k} + 2 \sum_{i=0}^{k-1} \sqrt{\rho_i}}.
	\end{align*}
	Firstly, because
	\begin{align*}
		b_k + 1 &= \frac{\left( \sum_{i=0}^{k-1} \sqrt{\rho_i} \right)^2 + 2 \sqrt{\rho_k} \sum_{i=0}^{k-1} \sqrt{\rho_i} + \rho_k}{\sqrt{\rho_k} \left( \sqrt{\rho_k} + 2 \sum_{i=0}^{k-1} \sqrt{\rho_i} \right)} \\
		&= \frac{\left( \sum_{i=0}^{k} \sqrt{\rho_i} \right)^2}{\sqrt{\rho_k} \left( \sqrt{\rho_k} + 2 \sum_{i=0}^{k-1} \sqrt{\rho_i} \right)},
	\end{align*}
	we have $\dfrac{c_k}{b_k + 1} = \rho_k$. Next, we show that $a_k \leqslant \dfrac{c_k}{2}$. Since $\sqrt{\rho_k} \geqslant 0$, we have
	\begin{align*}
		a_k &= \frac{\sqrt{\rho_k} \left( \sqrt{\rho_k} + 2 \sum_{i=0}^{k-1} \sqrt{\rho_i} \right)^2}{2 \left( \sqrt{\rho_k} + 2 \sum_{i=0}^{k-1} \sqrt{\rho_i} \right)} \\
		&\leqslant \frac{\sqrt{\rho_k} \left( 2 \sqrt{\rho_k} + 2 \sum_{i=0}^{k-1} \sqrt{\rho_i} \right)^2}{2 \left( \sqrt{\rho_k} + 2 \sum_{i=0}^{k-1} \sqrt{\rho_i} \right)} \\
		&= \frac{\sqrt{\rho_k} \left( \sum_{i=0}^{k} \sqrt{\rho_i} \right)^2}{2(\sqrt{\rho_k} + 2 \sum_{i=0}^{k-1} \sqrt{\rho_i})} \\
		&= c_k.
	\end{align*}
	Then by Theorem \ref{thm:discreterate}, the convergence rate of Algorithm \ref{al:spppa} is \[
	f(x_k)-f^*\leqslant\frac{\mathrm{dist}_L(x_0,\Omega)^2}{\left(\sum_{i=0}^{k-1}\sqrt{\rho_i}\right)^2},
	\]
	which is finer than the convergence rate of Algorithm \ref{al:appa} given in \eqref{eq:apparate}.
\end{example}
Next, we demonstrate the $o(1/A_k)$ convergence rate of Algorithm \ref{al:spppa}. Analogous to the approach taken in Section \ref{sec:2}, we establish this result under two distinct sets of assumptions.

\begin{theorem}
	\label{thm:oediscreterate}
	Let $\{x_k\}$, $\{\tilde{x}_{k+1}\}$, and $\{z_k\}$ be the sequences generated by Algorithm \ref{al:spppa}. Suppose the following conditions hold:
	\begin{itemize}
		\item $A_k = a_k b_k$,
		\item there exists a constant $d \in (0,1)$ such that $0 \leqslant A_{k+1} - A_k \leqslant d a_k$,
		\item there exists a positive constant $\gamma$ such that $a_k \geqslant \gamma A_k$.
	\end{itemize}
	Under these assumptions, it follows that
	\[
	\lim_{k \to \infty} A_k [f(x_k) - f^*] = 0.
	\]
\end{theorem}

\begin{proof}
	By \eqref{eq:valuesum}, the inequality $f(x_{k+1})-f(x^*)\leqslant\inner{\tilde{\nabla}f(x_{k+1}),x_{k+1}-x^*}$, and the assumption that there exists a constant $d\in(0,1)$ such that $0\leqslant A_{k+1}-A_k\leqslant da_k$, we have that $\lim\limits_{k\to\infty} a_k[f(x_k)-f^*]=0$. Because there exists a positive constant $\gamma$ such that $a_k\geqslant \gamma A_k$, we have\[
	\lim_{k\to\infty}A_k[f(x_k)-f^*]=0.
	\]
\end{proof}

\begin{theorem}
	\label{thm:okdiscreterate}
	Let $\{x_k\}$, $\{\tilde{x}_{k+1}\}$, and $\{z_k\}$ be the sequences generated by Algorithm \ref{al:spppa}. Suppose the following conditions are satisfied:
	\begin{itemize}
		\item $A_k = a_k b_k$,
		\item there exists a constant $d \in (0,1)$ such that $0 \leqslant A_{k+1} - A_k \leqslant d a_k$,
		\item $\sup_{k \geqslant 1} \dfrac{c_{k+1} - c_k}{a_k} < \infty$
		\item $\beta := \lim\limits_{k \to \infty} \dfrac{k}{b_k} \in (1, \infty)$,
		\item there exist two positive constants $\gamma_1 > 0.5$ and $\gamma_2$, and an integer $K$, such that $\gamma_1 a_k \leqslant c_k \leqslant \gamma_2 a_k$ for all $k \geqslant K$.
	\end{itemize}
	Under these conditions, it follows that
	\[
	\lim_{k \to \infty} A_k [f(x_k) - f^*] = 0.
	\]
\end{theorem}

\begin{proof}
	Step 1: Introduce an auxiliary sequence. Let $\{\G(k)\}$ be the auxiliary sequence defined as \[
	\G(k):=c_k[f(x_k)-f(x^*)]+\frac{1}{2}\norm{x_k-x^*}^2_L-\inner{x_k-x^*,z_k-x^*}_L.
	\]
	The difference of $\{\G(k)\}$ can be calculated as follows:\begin{align*}
		\G(k+1)-\G(k)=&\ (c_{k+1}-c_k)[f(x_{k+1})-f(x^*)]+c_k[f(x_{k+1})-f(x_k)]\\
		&\ +\inner{x_{k+1}-x_k,x_{k+1}-x^*}_L-\frac{1}{2}\norm{x_{k+1}-x_k}^2_L\\
		&\ -\inner{x_{k+1}-x^*,z_{k+1}-z_k}_L-\inner{x_{k+1}-x_k,z_k-x^*}_L\\
		\leqslant&\ -\left(b_k+\frac{1}{2}\right)\norm{x_{k+1}-x_k}^2_L+a_k\inner{\tilde{\nabla} f(x_{k+1}),x_{k+1}-x^*}\\
		&\ +(c_{k+1}-c_k)[f(x_{k+1})-f(x^*)].
	\end{align*}

	Step 2: Consider $\E_\alpha(k)=\E(k)+\alpha\G(k)$. Let $\alpha$ be an arbitrary constant that belongs to $\left(0,\frac{1-d}{1+\sup_{k\geqslant 1}\frac{c_{k+1}-c_k}{a_k}}\right)$. By previous calculations, we have\begin{align*}
		\E_\alpha(k+1)-\E_\alpha(k)	\leqslant &\ [A_{k+1}-A_k-(1-\alpha)a_k]\inner{\tilde{\nabla}f(x_{k+1}),x_{k+1}-x^*}\\
		&\ +\alpha(c_{k+1}-c_k)[f(x_{k+1})-f(x^*)]-\alpha\left(b_k+\frac{1}{2}\right)\norm{x_{k+1}-x_k}^2_L\\
		\leqslant&\ [A_{k+1}-A_k-(1-\alpha)a_k+\alpha(c_{k+1}-c_k)][f(x_{k+1})-f(x^*)]\\
		&\ -\alpha\left(b_k+\frac{1}{2}\right)\norm{x_{k+1}-x_k}^2_L\\
		\leqslant&\ -\alpha\left(b_k+\frac{1}{2}\right)\norm{x_{k+1}-x_k}^2_L,\qquad\forall k\geqslant 1.
	\end{align*}
	Thus, $\{\E_\alpha(k)\}_{k\geqslant 1}$ is non-increasing. In addition, since $\alpha<1$, we have\begin{equation}
		\label{eq:oxkx}
		\E_\alpha(k)= (A_t+\alpha c_t)[f(x_k)-f(x^*)]+\frac{\alpha-\alpha^2}{2}\norm{x_k-x^*}^2_{L}+\frac{1}{2}\norm{z_k-x^*-\alpha (x_k-x^*)}_L^2.
	\end{equation}
	Thus $\{\E_\alpha(k)\}$ is non-negative. The non-increasing property and non-negative of $\{\E_\alpha(k)\}$ imply that $\lim\limits_{k\to\infty}\E_\alpha(k)$ exists.

	Step 3: Prove the existence of $\lim\limits_{k\to\infty}\dfrac{1}{2}\norm{x_k-x^*}_L^2$. First, because the limits of both of $\E_\alpha(k)$ and $\E(k)$ exists, $\lim\limits_{k\to\infty}\G(k)$ exists. Also, because $f(x_k)-f^*\leqslant O(1/A_k)$ and $c_k\leqslant \gamma_2a_k$ for all $k\geqslant K$, we have\begin{equation}
		\lim_{k\to\infty}\frac{1}{2}\norm{x_k-x^*}^2_L-\inner{x_k-x^*,z_k-x^*}_L\text{ exists.}
		\label{eq:o1}
	\end{equation}
	Also, by \eqref{eq:oxkx}, we have that $\norm{x_k-x^*}^2_L$ is uniformly bounded. 
	
	Next, we consider the difference of $\dfrac{1}{2}\norm{x_k-x^*}^2$. Direct calculations show that\begin{align*}
		&\ \frac{1}{2}\norm{x_{k+1}-x^*}_L^2-\frac{1}{2}\norm{x_k-x^*}_L^2=\inner{x_{k+1}-x_k,x_{k+1}-x^*}_L-\frac{1}{2}\norm{x_{k+1}-x_k}_L^2\\
		= &\ \frac{1}{b_k}\inner{z_k-x_{k+1},x_{k+1}-x^*}_L-\frac{c_k}{b_k}\inner{\tilde{\nabla}f(x_{k+1}),x_{k+1}-x^*}-\frac{1}{2}\norm{x_{k+1}-x_k}_L^2\\
		=&\ -\left(\frac{1}{b_k}+\frac{1}{2}\right)\norm{x_{k+1}-x^*}_L^2+\frac{1}{b_k}\inner{x_{k+1}-x^*,z_{k+1}-x^*}_L\\
		&\ -\frac{c_k-a_k}{b_k}\inner{\tilde{\nabla} f(x_{k+1}),x_{k+1}-x^*}.
	\end{align*}
	Then we have\begin{align*}
		& \ \frac{1}{2}\norm{x_k-x^*}_L^2+\frac{k}{\beta}\left(\frac{1}{2}\norm{x_{k+1}-x^*}_L^2-\frac{1}{2}\norm{x_k-x^*}_L^2\right) \\
		= &\ \left(\frac{1}{2}-\frac{k}{b_k\beta}\right)\norm{x_{k+1}-x^*}_L^2+\frac{k}{b_k\beta}\inner{x_{k+1}-x^*,z_{k+1}-x^*}_L\\
		&\ -\frac{(c_k-a_k)(k-\beta)}{b_k\beta}\inner{\tilde{\nabla}f(x_{k+1}),x_{k+1}-x^*}-\frac{k-\beta}{2\beta}\norm{x_{k+1}-x_k}_L^2\\
		&\ +\frac{1}{b_k}\norm{x_k-x^*}_L^2-\frac{1}{b_k}\inner{x_{k+1}-x^*,z_{k+1}-x^*}_L.
	\end{align*}
	First, because $\norm{x_k-x^*}_L^2$ is uniformly bounded, $\lim\limits_{k\to\infty}\dfrac{1}{2}\norm{x_k-x^*}_L^2-\inner{x_k-x^*,z_k-x^*}_L$ exists, and $\lim\limits_{k\to\infty}kb_k^{-1}\in(1,\infty)$, we have that\[
	\lim_{k\to\infty}\frac{1}{b_k}\norm{x_k-x^*}_L^2-\frac{1}{b_k}\inner{x_{k+1}-x^*,z_{k+1}-x^*}_L=0.
	\] 
	Next, by \eqref{eq:valuesum} and $\lim\limits_{k\to\infty}\dfrac{k-\beta}{b_k}$ is finite, \[
	\lim_{k\to\infty}-\frac{(c_k-a_k)(k-\beta)}{b_k\beta}\inner{\tilde{\nabla}f(x_{k+1}),x_{k+1}-x^*}=0.
	\]
	Additionally, because $\E_\alpha(k+1)-\E_\alpha(k)\leqslant-\left(b_k+\dfrac{1}{2}\right)\norm{x_{k+1}-x_k}_L^2$ and $\lim\limits_{k\to\infty}kb_k^{-1}\in(1,\infty)$, we have\[
	\lim_{k\to\infty}-\frac{k-\beta}{2\beta}\norm{x_{k+1}-x_k}_L^2=0.
	\]
	Finally, owing to the existence of $\lim\limits_{k\to\infty}\dfrac{1}{2}\norm{x_k-x^*}_L^2-\inner{x_k-x^*,z_k-x^*}_L$, we have that 
	\begin{align*}
		&\ \lim_{k\to\infty}\left(\frac{1}{2}-\frac{k}{b_k\beta}\right)\norm{x_{k+1}-x^*}_L^2+\frac{k}{b_k\beta}\inner{x_{k+1}-x^*,z_{k+1}-x^*}_L \\
		= &\ \lim_{k\to\infty}\frac{1}{2}\norm{x_k-x^*}^2_L-\inner{x_k-x^*,z_k-x^*}_L.
	\end{align*}
	Then by Lemma \ref{lem:klimitexist}, we can show that $\lim\limits_{k\to\infty}\dfrac{1}{2}\norm{x_k-x^*}_L^2$ exists.

	Step 4: Demonstrate the $o(1/A_k)$ convergence rate. Owing to the existence of $\lim\limits_{k\to\infty}\dfrac{1}{2}\norm{x_k-x^*}_L^2$, we obtain that $\lim\limits_{k\to\infty}\inner{x_k-x^*,z_k-x^*}_L$ exists. Since $z_k=b_k(x_{k+1}-x_k)+c_kL^{-1}\tilde{\nabla}f(x_{k+1})+x_{k+1}$, we have that\begin{equation}
		\lim_{k\to\infty}\inner{(b_k+1)(x_{k+1}-x_k)+c_kL^{-1}\tilde{\nabla}f(x_{k+1}),x_k-x^*}_L\text{ exists.}
		\label{eq:o2}
	\end{equation}
	Also, because of \eqref{eq:o1}, \eqref{eq:o2}, and the existence of $\lim\limits_{k\to\infty}\E(k)$, we have \begin{equation}
		\lim_{k\to\infty}A_k[f(x_k)-f^*]+\frac{1}{2}\norm{(b_k+1)(x_{k+1}-x_k)+c_kL^{-1}\tilde{\nabla}f(x_{k+1})}_L^2\text{ exists.}
	\end{equation}
	
	Next, because $\{\E_\alpha(k)\}_{k\geqslant 1}$ is non-increasing and non-negative and $\E_\alpha(k+1)-\E_\alpha(k)\leqslant-\alpha\left(b_k+\dfrac{1}{2}\right)\norm{x_{k+1}-x_k}_L^2$ for all $k\geqslant 1$, we have\begin{equation}
		\label{eq:o3}
		\sum_{k=1}^\infty\left(b_k+\frac{1}{2}\right)\norm{x_{k+1}-x_k}_L^2<\infty.
	\end{equation}
	By Cauchy-Schwarz inequality, we have\begin{equation}
		\label{eq:o4}
		\frac{1}{2}\norm{(b_k+1)(x_{k+1}-x_k)+L^{-1}c_k\tilde{\nabla}f(x_{k+1})}_L^2\leqslant(b_k+1)^2\norm{x_{k+1}-x_k}_L^2+c_k^2\norm{\tilde{\nabla} f(x_{k+1})}_{L^{-1}}^2.
	\end{equation}
	Based on \eqref{eq:o1}-\eqref{eq:o4}, \eqref{eq:valuesum}, \eqref{eq:normsum}, and $\gamma_1a_k\leqslant c_k\leqslant \gamma_2a_k$, we have that\begin{align*}
		\lim_{k\to\infty}A_k[f(x_k)-f^*]= &\ \lim_{k\to\infty}A_k[f(x_k)-f^*]+\frac{1}{2}\norm{(b_k+1)(x_{k+1}-x_k)+L^{-1}c_k\tilde{\nabla}f(x_{k+1})}_L^2\\
		=&\ 0.
	\end{align*}
\end{proof}

At the end of this section, we present several instances of Algorithm \ref{al:spppa} that satisfy the conditions given in Theorem \ref{thm:oediscreterate} or Theorem \ref{thm:okdiscreterate}.
\begin{example}
	\label{ex:constant}
	Firstly, we propose an instance of Algorithm \ref{al:spppa} such that $\dfrac{c_k}{b_k + 1}$ is constant. Given $r\geqslant 2$, we set\begin{align*}
		A_k &= \frac{ck(k+r)}{2r^2}, \\
		a_k &= \frac{c(k+r)}{r},\\
		b_k &= \frac{k}{r},\\
		c_k &= \frac{c(k+r)}{r}.
	\end{align*}
	Direct calculation shows that\[
	A_{k+1}-A_k=\frac{c(2k+r+1)}{r^2}\leqslant\frac{2}{r}a_k.
	\]
	Then by Theorem \ref{thm:okdiscreterate}, if $r>2$, the convergence rate of Algorithm \ref{al:spppa} is essentially $o(1/k^2)$. If $r=2$, the convergence rate of Algorithm \ref{al:spppa} is $O(1/k^2)$.
\end{example}

\begin{example}
	Next, we present an instance of Algorithm \ref{al:spppa} with a $p$-th order convergence rate, where $p \geqslant 1$ is an integer. Here, $k^{(p)}$ denotes $\prod_{i=0}^{p-1} (k + i)$. Given $d \in (0, 1]$, we set
	\begin{align*}
		A_k &= k^{(p)}, \\
		a_k &= p d^{-1} (k + 1)^{(p-1)}, \\
		b_k &= d p^{-1} k, \\
		c_k &= a_k.
	\end{align*}
	By Theorem \ref{thm:discreterate}, the convergence rate of Algorithm \ref{al:spppa} is $O(1/k^{(p)})$. Moreover, if $d < 1$, all assumptions stated in Theorem \ref{thm:okdiscreterate} are satisfied, implying the convergence rate is essentially $o(1/k^{(p)})$.
\end{example}

\begin{example}
	Subsequently, we introduce an instance of Algorithm \ref{al:spppa} with an exponential convergence rate. Given two constants $\rho > 1$ and $d \in (0, 1]$, we set
	\begin{align*}
		A_k &= \rho^k, \\
		a_k &= d^{-1} (\rho - 1) \rho^k, \\
		b_k &= d (\rho - 1)^{-1}, \\
		c_k &= a_k.
	\end{align*}
	By Theorem \ref{thm:discreterate}, the convergence rate of Algorithm \ref{al:spppa} is $O(\rho^{-k})$. Furthermore, if $d < 1$, all assumptions given in Theorem \ref{thm:oediscreterate} are satisfied, implying a convergence rate of $o(\rho^{-k})$.
\end{example}

\section{Application: Symplectic Augmented Lagrangian Method}
\label{sec:b}
In this section, we derive the symplectic augmented Lagrangian method via Algorithm \ref{al:spppa}. First, we give a brief introduction to Lagrangian duality theory. To acquire further theory about Lagrangian duality theory, we refer to \cite{rockafellar1970,bonnas2000}. For the optimization problem \eqref{eq:main}, the corresponding perturbation problem is given as follows:\begin{equation}
	\min_x\varphi(x,u),\tag{$P_u$}
	\label{eq:perturbation}
\end{equation}
where $H'$ is an additional real Hilbert space, $\varphi: H\times H'\to\mathbb{R}$ is the perturbation function of $f$ such that $\varphi(x,0)=f(x)$. The Lagrangian function is given as follow:
\begin{equation}
	L(x,\lambda)=\inf_u\varphi(x,u)-\inner{\lambda,u}.
	\label{eq:lagrangian}
\end{equation}
The Lagrangian duality problem respect to the primal problem \eqref{eq:main} is defined as\begin{equation}
	\max_\lambda\min_x L(x,\lambda).
	\tag{$LD$}
	\label{eq:LD}
\end{equation}
Because \eqref{eq:LD} can be viewed as the optimization problem respect to the objective function $-\inf_x L(x,\cdot)$, Algorithm \ref{al:spppa} can be used to solve \eqref{eq:LD}. By applying Algorithm \ref{al:spppa} to \eqref{eq:LD}, we obtain\begin{align}
	\tilde{\lambda}_{k+1} &= \frac{b_k}{b_k+1}x_k+\frac{1}{b_k+1}z_k, \\
	\lambda_{k+1} &= \argmin_\lambda\left\{-\inf_x L(x,\lambda)+\frac{b_k+1}{2c_k}\norm{\lambda-\tilde{\lambda}_{k+1}}^2_L\right\},\label{eq:minmax1}\\
	z_{k+1}&=z_k+\frac{a_k(b_k+1)}{c_k}(x_{k+1}-\tilde{x}_{k+1}).
\end{align}
Next, we simplify \eqref{eq:minmax1}. First, \eqref{eq:minmax1} can be transformed into the following max-min problem:\begin{equation}
	\max_\lambda\min_{x,u}\varphi(x,u)-\inner{\lambda,u}-\frac{b_k+1}{2c_k}\norm{\lambda-\tilde{\lambda}_{k+1}}^2.
	\label{eq:minmax2}
\end{equation}
Owing to Theorem 34.3 in \cite{rockafellar1970}, if $\varphi$ is a proper convex function, \eqref{eq:minmax2} is equivalent to\begin{equation}
	\min_{x,u}\max_\lambda\varphi(x,u)-\inner{\lambda,u}-\frac{b_k+1}{2c_k}\norm{\lambda-\tilde{\lambda}_{k+1}}^2.
	\label{eq:minmax3}
\end{equation}
It is obvious that the solution of \eqref{eq:minmax3} can be given as follows:\begin{align}
	(x_{k+1},u_{k+1}) &= \argmin_{x,u}\varphi(x,u)-\inner{\tilde{\lambda}_{k+1},u}+\frac{c_k}{2(b_k+1)}\norm{u}^2_{L^{-1}}, \\
	\lambda_{k+1} &= \tilde{\lambda}_{k+1}-\frac{c_k}{b_k+1}L^{-1}u_{k+1}.
\end{align}
In conclusion, by replacing $L^{-1}$ by $L$, we obtain the following algorithm.
\begin{algorithm}
	\caption{Symplectic Augmented Lagrangian Method}
	\label{al:salm}
	\textbf{Initialization:} $\lambda_0$, $z_0=\lambda_0$\;
	\For{$k=0, 1, \cdots$}{
		Choose $a_k$, $b_k$, $c_k$\;
		$\tilde{\lambda}_{k+1}=\dfrac{b_k}{b_k+1}\lambda_k+\dfrac{1}{b_k+1}z_k$\;
		$(x_{k+1},u_{k+1}) = \argmin_{x,u}\varphi(x,u)-\inner{\tilde{\lambda}_{k+1},u}+\dfrac{c_k}{2(b_k+1)}\norm{u}^2_L$\;
		$\lambda_{k+1}=\tilde{\lambda}_{k+1}-\dfrac{c_k}{b_k+1}Lu_{k+1}$\;
		$z_{k+1}=z_k-a_kLu_{k+1}$\;
	}
\end{algorithm}

The reason why Algorithm \ref{al:salm} is called Symplectic Augmented Lagrangian Method is that \[
L_\rho(x,\lambda)=\inf_u \varphi(x,u)-\inner{\lambda,u}+\dfrac{\rho}{2}\norm{u}^2
\]
is the augmented Lagrangian function. 
\begin{example}[Convex programming with linear equality constraints]
	Here, we consider the following problem:\begin{equation}
		\begin{gathered}
			\min_x f(x),\\
			s. t.\quad Ax=b,
		\end{gathered}
		\label{eq:linearconstraints}
	\end{equation}
	where $A\in\mathbb{R}^{m\times n}$ and $f\in\Gamma_0(\mathbb{R}^n)$. Here we assume The perturbation function of \eqref{eq:linearconstraints} is given as follows:\begin{equation}
		\varphi(x,u)=f(x)+I(Ax-b=u).
	\end{equation}
	It is easy to verify that $\varphi$ is closed proper convex if there exists $x$ such that $x$ belongs to relative interior of effective domain of $f$ and $Ax=b$. Also, the Lagrangian function can be transformed as follows:\begin{align*}
		L(x,\lambda) &= \inf_u f(x)+I(Ax-b=u)-\inner{\lambda,u} \\
		&= f(x)-\inner{\lambda,Ax-b}.
	\end{align*}
	Also, the iteration formula of Algorithm \ref{al:salm} applied to \eqref{eq:linearconstraints} can be alternated by:\begin{align*}
		\tilde{\lambda}_{k+1}&=\dfrac{b_k}{b_k+1}\lambda_k+\dfrac{1}{b_k+1}z_k,\\
		x_{k+1} &= \argmin_{x}f(x)-\inner{\tilde{\lambda}_{k+1},Ax-b}+\dfrac{c_k}{2(b_k+1)}\norm{Ax-b}^2_L,\\
		\lambda_{k+1}&=\tilde{\lambda}_{k+1}-\dfrac{c_k}{b_k+1}L(Ax_{k+1}-b),\\
		z_{k+1}&=z_k-a_kL(Ax_{k+1}-b).
	\end{align*}
\end{example}

\begin{lemma}
	Let $\{x_{k+1}\}$, $\{\lambda_k\}$, $\{\tilde{\lambda}_{k+1}\}$, and $\{z_k\}$ be the sequences generated by Algorithm \ref{al:salm}. If $\varphi(\cdot,\cdot)$ is closed proper convex function, then we have\[
	\inf_x L(x,\lambda_{k+1})=L(x_{k+1},\lambda_{k+1})=\varphi(x_{k+1},u_{k+1})-\inner{\lambda_{k+1},u_{k+1}}.
	\]
\end{lemma}

\begin{proof}
	Since $(x_{k+1},u_{k+1})$ is the optimal solution of optimization subproblem, then we have
	\begin{align*}
		0&\in\partial\varphi(x_{k+1},u_{k+1})-\tilde{\lambda}_{k+1}+\frac{c_k}{b_k+1}u_{k+1}\\
		&= \partial\varphi(x_{k+1},u_{k+1})-\lambda_{k+1}.
	\end{align*}
	Then by the definition of augmented Lagrangian function, we have $L(x_{k+1},\lambda_{k+1})=\varphi(x_{k+1},u_{k+1})-\inner{\lambda_{k+1},u_{k+1}}$. Next, because $\inf L(x,\lambda_{k+1})=\inf_{x,u}\varphi(x,u)-\inner{\lambda_{k+1},u}$,  we have $\inf_xL(x,\lambda_{k+1})=L(x_{k+1},\lambda_{k+1})$.
\end{proof}

\begin{corollary}
	Let $\{x_{k+1}\}$, $\{\lambda_k\}$, $\{\tilde{\lambda}_{k+1}\}$, and $\{z_k\}$ be the sequences generated by Algorithm \ref{al:salm}. If $\varphi(\cdot,\cdot)$ is closed proper convex function and \eqref{eq:discretecondition} holds, then for all saddle point $(x^*,\lambda^*)$ of Lagrangian function $L$, i. e. \[
	L(x^*,\lambda)\leqslant L(x^*,\lambda^*)\leqslant L(x,\lambda^*),\qquad\forall x, \lambda,
	\]
	we have\[
	L(x^*,\lambda^*)-L(x_k,\lambda_k)\leqslant\frac{A_0}{A_k}[L(x^*,\lambda^*)-\inf_x L(x,\lambda_0)]+\frac{1}{2A_k}\dist_{L^{-1}}(x_0,\Omega)^2,
	\]
	where $\Omega$ denote the solution set of \eqref{eq:LD}. In addition, we have\[
	\sum_{k=0}^\infty\left(a_kc_k-\frac{a_k^2}{2}\right)\norm{u_{k+1}}^2_L\leqslant A_0[L(x^*,\lambda^*)-\inf_xL(x,\lambda_0)]+\frac{1}{2}\dist_{L^{-1}}(x_0,\Omega)^2.
	\]
\end{corollary}

Our convergence result of Algorithm \ref{al:salm} extends the convergence result given in \cite{he2010}.

\section{Discussion}

In this paper, we first investigate an ODE system given by \eqref{eq:continuousmain}. Utilizing the Lyapunov function technique, we establish that the convergence rate of \eqref{eq:continuousmain} is \( O(1/A_t) \). Subsequently, by introducing additional assumptions and refining the Lyapunov function, we demonstrate the \( o(1/A_t) \) convergence rate of \eqref{eq:continuousmain}. Building upon these findings, we apply the Symplectic Euler method to \eqref{eq:continuousmain} to derive a novel proximal point algorithm for convex minimization, called the Symplectic Proximal Point Algorithm. Inspired by the analysis conducted for \eqref{eq:continuousmain}, we initially employ a discrete-time Lyapunov function to prove the \( O(1/A_k) \) convergence rate of Symplectic Proximal Point Algorithm. Additionally, by introducing further assumptions and modifying the Lyapunov function, we are able to prove the \( o(1/A_k) \) convergence rate of Symplectic Proximal Point Algorithm.

Several valuable research questions remain open. Firstly, the proximal point algorithm has a more generalized form known as the Bregman proximal point algorithm. In \cite{lei2022}, the authors drew inspiration from the NAG method and proposed a novel accelerated Bregman proximal point algorithm. It would be of interest to develop a new ODE system for the accelerated Bregman proximal point algorithm and employ the Symplectic Euler method to derive a novel accelerated Bregman Proximal Point Algorithm. Furthermore, the optimization subproblems that arise in the Symplectic Proximal Point Algorithm may not always be solvable exactly. Therefore, developing an inexact variant of the Symplectic Proximal Point Algorithm is essential for practical applications.

\appendix
\section{Auxiliary results}
In this section, we present some auxiliary results used in the proof for our results. 
\begin{lemma}[Lemma 5.1 in \cite{abbas2014}]
	\label{lem:existencelimit}
	Let $\varepsilon>0$ be a constant. Suppose that $f\in[\varepsilon,\infty)$ is locally absolutely continuous, bounded from
	below, and there exists $g\in\mathbb{L}^1([\varepsilon,\infty))$ such that for almost every $t\geqslant\varepsilon$\[
	\frac{\dd}{\dd t}f(t)\leqslant g(t).
	\]
	Then the limit $\lim\limits_{t\to\infty}f(t)$ exists.
\end{lemma}

\begin{lemma}[Lemma A.2 in \cite{attouch2016fast}]
	\label{lem:tlimitexist}
	Let $a>0$ and $q: [t_0,\infty)$ be a continuously differentiable function such that\[
	\lim_{t\to\infty}q(t)+\frac{t}{a}\dot{q}(t)=l\in\hh.
	\]
	Then $\lim\limits_{t\to\infty}q(t)=l$.
\end{lemma}
\begin{lemma}[Lemma A.5 in \cite{bot2023ogda}]
	\label{lem:klimitexist}
	Let $a\geqslant 1$ and $\{q_k\}$ be a bounded sequence in $\hh$ such that \[
	\lim_{k\to\infty}q_k+\frac{k}{a}(q_{k+1}-q_k)=l\in\hh.
	\]
	Then it holds $\lim\limits_{k\to\infty}q_k=l$.
\end{lemma}

\bibliographystyle{amsplain}
\bibliography{reference}

\end{document}